\documentclass[11pt, reqno]{amsart}
\usepackage[all,tips]{xy}
\usepackage{latexsym, amsfonts, amsmath, amssymb, mathrsfs, graphicx, xcolor, ytableau, hyperref, float}
\usepackage[justification=centering]{caption}
\allowdisplaybreaks
\usepackage{centernot}
\usepackage{tikz}
\usepackage{pgfplots}

\definecolor{red}{rgb}{1,0,0}
\definecolor{magenta}{rgb}{1,0,1}
\definecolor{dartmouthgreen}{rgb}{0.05, 0.5, 0.06}
\definecolor{purple(x11)}{rgb}{0.63,0.36,0.94}
\definecolor{turquoise}{rgb}{0.25, 0.87, 0.81}
\newtheorem{theorem}{Theorem}[section]
\newtheorem{lemma}[theorem]{Lemma}
\newtheorem{proposition}[theorem]{Proposition}
\newtheorem{corollary}[theorem]{Corollary}

\newtheorem{definition}[theorem]{Definition}

\newtheorem*{question}{Question}

\theoremstyle{definition}
\newtheorem{remark}[theorem]{Remark}

\newcommand{\Log}{\mathrm{Log}}

\newcommand{\Cal}[1]{\ensuremath{\mathcal{#1}}}

\newcommand{\lp}{\left(}
\newcommand{\rp}{\right)}

\usepackage[textheight=8.75in, textwidth=6.75in]{geometry}

\def\C{{\mathbb C}}

\def\N{{\mathbb N}}
\def\Z{{\mathbb Z}}
\def\Q{{\mathbb Q}}

\def\O_K{{\Cal{O}_{K}}}
\def\O_F{{\Cal{O}_{F}}}
\def\N_F{{\Cal{N}_{F/\Q}}}

\def\O_K{{\Cal{O}_{K}}}
\def\O_F{{\Cal{O}_{F}}}
\def\N_F{{\Cal{N}_{F/\Q}}}

\numberwithin{equation}{section}
\numberwithin{theorem}{section}

\title{Quasimodular forms, partitions, and primes in intervals}

\author{Kevin Allen}
\address{School of Mathematical Sciences, University College Cork, Cork, Ireland}
\email{kallen@ucc.ie}

\author{William Craig}
\address{Department of Mathematics, United States Naval Academy, 572C Holloway Road
Mail Stop 9E. Annapolis, MD 21402}
\email{wcraig@usna.edu}


\begin{document} 
	
\begin{abstract} 
Recent developments in theories of partitions, quasimodular forms, multiple
zeta values and $q$-multiple zeta values have led to connections between partitions
and primes. Central to this progress are prime-detecting quasimodular forms,
which are $q$-series whose non-negative coefficients vanish if and only if the input is prime. In this paper, we use properties of these quasimodular forms to motivate a proof of the infinitude of primes, to construct a novel partition theoretic factorization algorithm, and to show that the distribution of these coefficients relate to existence of primes in intervals. We also prove analogs of these results for primes in arithmetic progressions and discuss potential extensions of our results, including to Linnik's problem.
\end{abstract}

\maketitle

\section{Introduction}

Throughout its history, number theory has been woven from surprising connections between apparently unrelated ideas. From the use of $L$-functions to prove the infinitude of primes in arithmetic progressions, to the role of modularity and Riemann surfaces in the combinatorial study of partitions, to the use of Galois representations to resolve difficult Diophantine equations, the story of number theory is an ever-expanding web of interplay between geometry, arithmetic, analysis, probability, and combinatorics.

In this work, we consider a rising thread of the interconnectedness of number theory. This thread emerges through a strange connection between additive and multiplicative number theory, first considered in 1919 by MacMahon \cite{MacMahon}. His goal was to generalize the sum-of-divisor functions
\begin{align*}
    \sigma_k(n) := \sum_{d|n} d^k
\end{align*}
in an additive direction; instead of splitting $d|n$ into multi-factorizations $d_1 \cdots d_m | n$, as is done in classical multiplicative number theory, he sought a generalization of $d|n$ using integer partitions as its building block. From the observation
\begin{align*}
    n = md \iff n = \underbrace{d + d + \cdots + d}_{m \textrm{ times}},
\end{align*}
MacMahon took as his generalization of dividing $n$ into $a$ distinct factors the idea of finding the partitions of $n$ into $a$ distinct part sizes. Thus, MacMahon studied partitions
\begin{align*}
    n = m_1 d_1 + \dots m_a d_a
\end{align*}
and constructed divisor sums related to them, which are defined for $a \geq 1$ by
\begin{equation*}
M_a(n) = \sum_{\substack{0<d_1<\cdots<d_a\\n=m_1d_1+\cdots+m_ad_a}}m_1m_2\cdots m_a.
\end{equation*}

Subsequently, MacMahon's work has been greatly generalized in the theory of $q$-multiple zeta values. For instance, Bachmann and K\"{u}hn \cite{BachmannKuhn} study $q$-multiple zeta values for vectors $\vec s = (s_1, \dots, s_a)$ with Fourier coefficients defined  by
\begin{align*}
    M_{\vec s} (n) := \sum_{\substack{0<d_1<\cdots<d_a\\n=m_1d_1+\cdots+m_ad_a}} m_1^{s_1} \cdots m_a^{s_a},
\end{align*}
so that their generating function is
\begin{align*}
    \mathcal U_{\vec s}(q) = \sum_{n_1 > \dots > n_k > 0} \dfrac{q^{n_1 + \dots + n_k} P_{s_1}(q^{n_1}) \cdots P_{s_k}(q^{n_k})}{\lp 1 - q^{n_1} \rp^{s_1+1} \cdots \lp 1 - q^{n_k} \rp^{s_k+1}},
\end{align*}
where $P_n(x)$ are the Eulerian polynomials defined by
\begin{equation*}
    \frac{xP_n(x)}{(1-x)^{n+1}}:=\sum_{k\geq 1}k^nx^n.
\end{equation*}

MacMahon's work has inspired a great deal of number-theoretic work in recent years, e.g. \cite{AmdeberhanAndrewsTauraso,AmdeberhanOnoSingh,AndrewsRose,BringmannCraigIttersumPandey,OnoSingh,Zagier}. In particular, the second author and collaborators \cite{CraigIttersumOno} extended MacMahon's thread by proving that MacMahon's functions $M_{\vec s}(n)$ can be combined in infinitely many distinctive ways to detect the prime numbers, in the sense that there are infinitely many expressions 
\begin{align*}
    \sum_{\vec s} c_{\vec s} M_{\vec s}(n)
\end{align*}
where $\vec s$ is summed over vectors of non-negative integers of bounded length (not necessarily all the same length), that vanish for $n \geq 2$ if and only if $n$ is prime. Since the values of $M_{\vec s}(n)$ can be calculated purely as an exercise in the additive enumeration of partitions of integers, these functions draw a direct connection between partition theory and prime number theory. Since this work, a number of authors have investigated generalizations of this prime detection problem and related problems \cite{Alkan, Craig, Gomez, IttersumMauthOnoSingh, KaneKrishnamoorthyLau, KangMatsusakaShin}. We also refer to Schneider's work linking additive and multiplicative number theory \cite{Schneider17,SchneiderThesis,Schneider16}, which has as its object the generalization of multiplicative functions on $\Z$ to a multiplicative theory of functions on the set of partitions. 

In this work, we revisit the original prime-detecting series constructed in \cite{CraigIttersumOno} from a different point of view. The impetus for this investigation is to answer the following question: Can one prove that there are infinitely many prime numbers using as the key tool the prime-detecting quasimodular forms? In this work, we answer this question in the affirmative by studying  not the vanishing properties of these forms at the prime coefficients, but the non-vanishing coefficients at composite inputs.

To explain further, we consider the prime-detecting quasimodular forms discovered by Leli\`evre \cite{Lelievre} and studied closely in \cite{CraigIttersumOno} defined for odd integers $\ell>k>0$ by
\begin{align}\label{fkl}
    f_{k,\ell}(\tau) := \lp D^\ell + 1 \rp G_{k+1}(\tau) - \lp D^k + 1 \rp G_{\ell+1}(\tau) = \sum_{n \geq 0} c_{k,\ell}(n) q^n, \ \ \ \ \ q = e^{2\pi i \tau}, \ \tau \in \mathbb{H},
\end{align}
where $G_k$ are the classical Eisenstein series and $D = q\frac{d}{dq}$ is the canonical differential operator in modular form theory (see Section \ref{Sec: Prelims} for definitions and details).  \noindent For example, we have the Fourier expansion
\begin{equation*}
f_{3,7}(q) = \dfrac{1}{480} + 122760q^4 + 9295440q^6 + 142524360q^8 + 127543680q^9 + 1251668880q^{10} + \ \dots
\end{equation*}
Leli\`evre showed that these forms are prime-detecting \cite[Corollary 1]{Lelievre}. For a background on modular forms and quasimodular forms, we refer the reader to \cite{DiamondShurman, Royer}.

One of the main goals of this paper is to analyze the asymptotic distribution of the coefficients $c_{k,\ell}(n)$. The behavior of the coefficients $c_{3,7}(n)$ is illustrated in Figure \ref{fig:c37}.

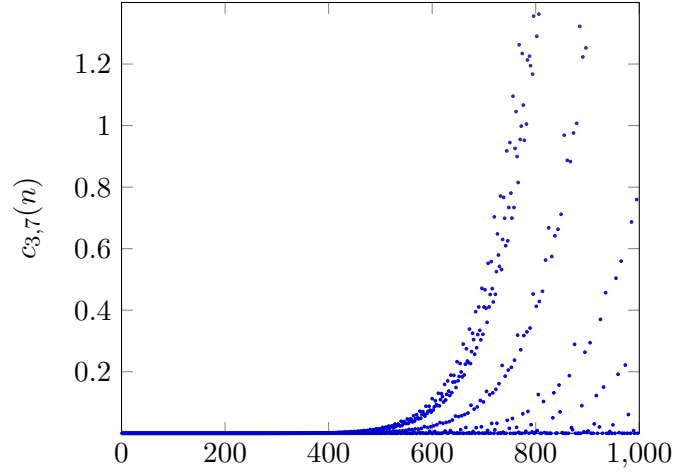
\begin{figure}[h]
    \centering

    \begin{tikzpicture}
    \begin{axis}[
        title={},
        xlabel={},
        ylabel={$c_{3,7}(n)$},
        xmin=0, xmax=1000,
        ymin=0, ymax=14000000000000000000000000000,
        xtick={0,200,400,600,800,1000},
        ytick={2000000000000000000000000000,4000000000000000000000000000,6000000000000000000000000000,8000000000000000000000000000,10000000000000000000000000000,12000000000000000000000000000},
        scaled y ticks=manual:{}{\pgfmathparse{#1/10000000000000000000000000000}}
        ]

        \addplot+[
        only marks,
        scatter,
        mark=*,
        black,
        mark size=0.5pt
        ]
        table{f37data.dat};
    \end{axis}
\end{tikzpicture}
    
    \caption{Graph of $c_{3,7}(n)$ $(\times 10^{28})$ for $0\leq n \leq 1000$.}
    \label{fig:c37}
\end{figure}

In modular form theory, from which these objects emerge, there are a few types of distributions one typically sees when studying Fourier coefficients. For Hecke eigenforms, the distributions are governed by Sato--Tate distributions for prime coefficients and Hecke theory for all others. Forms with complex multiplication have very large numbers of vanishing coefficients. For typical non-cuspidal forms like the Eisenstein series of weight $k$, coefficients grow roughly like $n^k$, with some small variations based on smaller divisors of $n$. For weakly holomorphic forms, coefficients grow subexponentially, as is the case with the partition function $p(n)$.

The coefficients of $f_{3,7}$, and indeed every prime-detecting quasimodular form, form a different pattern entirely. Apart from the zeros at primes, there are clearly visible, discrete strips in the graph of $c_{k,\ell}(n)$, which emerge very quickly from the noise at lower values of $n$. The underlying phenomenon explaining this unusual structure is that the coefficients of prime-detecting quasimodular forms at composite inputs exhibit growth like a constant power of $n$ multiplied by a decay factor depending explicitly on the smallest prime dividing $n$. 

The focus of this paper is to initiate an exploration of the distribution of values of $c_{k,\ell}(n)$ and its connection to the distribution of primes. Our first result, shows how the transformation of $f_{k,\ell}$ into its associated $L$-function yields a rapid proof of the infinitude of the primes.

\begin{theorem} \label{Thm: Primes}
    There are infinitely many primes.
\end{theorem}

The distribution of values, however, contains much more refined information about the primes. In our work, we provide two further extensions of this numerical observation into deeper aspects of prime number theory. Our first application is to bound the smallest prime factor of a large integer in terms of the values of $c_{k,\ell}(n)$.

\begin{theorem} \label{Thm: Quick Factorization Range}
    Let $N > 3$ be a composite integer, and let $\ell > k \geq 3$ be odd integers. Then either $N$ has a prime factor $p < \log(N)$, or has a prime factor $p$ in the interval
    \begin{align*}
        \left[ \lp \dfrac{4 c_{k,\ell}(N)}{3 N^{k+\ell}} \rp^{-1/k}, \lp \lp \dfrac{4 c_{k,\ell}(N)}{3 N^{k+\ell}} \rp^{-1/k} + 1 \rp \lp \dfrac{8}{3} \omega(N) \rp^{1/k} \right],
    \end{align*}
    where $\omega(N)$ is the number of distinct primes dividing $N$.
\end{theorem}

\begin{remark}
        As a corollary in the spirit of \cite{CraigIttersumOno}, we now have a partition-theoretic algorithm for the factorization of large integers. We describe this procedure in Section \ref{Sec: Appendix A} and clarify issues of circularity which must be ironed out for the internal coherence of such an algorithm. We also give examples in Appendix \ref{Sec: Appendix B} which demonstrate that both the upper and lower bounds in Theorem \ref{Thm: Quick Factorization Range} are not vacuous.
\end{remark}

Our second application will use Theorem \ref{Thm: Quick Factorization Range} to prove results that connect the distribution of values of $c_{k,\ell}(n)$ directly to the existence of primes in intervals.

\begin{theorem} \label{Thm: Primes in intervals}
    The following are true:
    \begin{enumerate}
        \item Assume that $k>1$ is an odd integer and that for $c>1$ and $n\gg_c0$, every interval $(n,cn]$ contains at least one prime. Then for $\delta>0$ sufficiently small and $\varepsilon > 2\delta + 3\delta^2$, there are infinitely many values in the set $\{ c_{k,\ell}(N)/N^{k+\ell} \}_{N,\ell}$ that intersect the interval $(\delta,\varepsilon)$.
        \item Assume that $k>1$ is an odd integer and that the sets
        \begin{align*}
            S_k = \bigg\{ \dfrac{c_{k,\ell}(N)}{N^{k+\ell}} : \ell>k \textrm{ odd, } \omega(N) \leq C(k) \bigg\},
        \end{align*}
        for some function $C(k)$ satisfying $C(k)^{1/k} \to 1$ as $k \to \infty$, intersect all intervals $(\delta,\varepsilon)$ with $\varepsilon > \delta + \delta^{(k+1)/k}$ and $\delta$ small.
        Then every interval $(n,cn]$ for $c>1$ and $n \gg_k 0$ contains at least one prime.
    \end{enumerate}
\end{theorem}

\begin{remark}
     Since the existence of primes in the intervals $(x,cx)$ is a straightforward consequence of the prime number theorem, we conclude that $\{ c_{k,\ell}(N)/N^{k+\ell} \}_{N,\ell}$ does in fact have values in these intervals. However, (2) is not quite a converse to (1), in particular because the proof of (1) does not enforce a limitation on $\omega(N)$. It may be possible to strengthen (1) by utilizing the flexibility of $\ell>k$ more effectively. It would also be of interest whether improvements to Theorem \ref{Thm: Quick Factorization Range} might yield access to results on primes in short intervals. Variations of (1) could also be proven using the existence of many primes in an interval $(n,2n]$, which can be deduced, e.g., from Erd\H{o}s' elementary proof of Bertrand's postulate \cite{Erdos}.
\end{remark}

In light of Theorems \ref{Thm: Primes}, \ref{Thm: Quick Factorization Range} and \ref{Thm: Primes in intervals}, it is natural to inquire to what degree similar techniques could be used to study primes in arithmetic progressions. If one simply replaces divisor sums by restricted divisor sums over all divisors $d \equiv a \pmod{b}$, one does retain connections to modular form theory (via standard character twists). This approach, however, will fail to properly account for the primes in progressions, because our techniques would require that every integer $n \equiv a \pmod{b}$ contains a prime divisor in the same congruence class, which does not hold in general.

There is a modification of this idea whereby similar results can be proven, under an additional assumption. Such a result requires introducing the following restricted divisor sums.

\begin{definition}
    Let $1 \leq a \leq b$ be coprime integers and let $\ell > k \geq 0$ be integers. Let \footnote{We say $1 \in \mathcal N_{a,b}$ by convention.}
    \begin{align*}
        \mathcal N_{a,b} := \{ n \in \mathbb N : \text{Any prime divisor } p \text{ of } n \text{ satisfies } p \equiv a \pmod{b} \}, 
    \end{align*}
    and define
    \begin{align*}
        \sigma_k^{a,b}(n) := \sum_{\substack{d|n \\ d \in \mathcal N_{a,b}}} d^k
    \end{align*}
    and
    \begin{align*}
        c_{k,\ell}^{a,b}(n) := (1+n^\ell) \sigma_k^{a,b}(n) - (1+n^k)\sigma_\ell^{a,b}(n).
    \end{align*}
\end{definition}

We will show that these functions behave quite similarly to the prime-detecting quasimodular forms provided the existence of at least one prime $p \equiv a \pmod{b}$. Such a hypothesis, if established for every permissible progression, already implies the infinitude of primes in all permissible progressions, since any arithmetic progression $n \equiv a \pmod{b}$ has infinitely many disjoint subprogressions with larger common differences. Our proof, however, assumes only that the singular residue class $a \pmod{b}$ contains a prime, and so is a much weaker hypothesis. The tools involved in the proof are also very elementary; they appear to require less knowledge of analysis than the famous elementary proofs of Selberg \cite{SelbergInf,SelbergPNT} for Dirichlet's theorem and the prime number theorem for arithmetic progressions.

We now briefly state the analogs of Theorems \ref{Thm: Quick Factorization Range} and \ref{Thm: Primes in intervals} for primes in arithmetic progressions.

\begin{theorem} \label{Thm: Quick Factorization Range Progressions}
    Let $N > 3$ be a composite integer, and let $\ell > k \geq 3$ be odd integers and $2 \leq a \leq b+1$ be such that $a$, $b$ are coprime. Assume there is at least one prime congruent to $a$ modulo $b$. Then either $N$ has a prime factor $p < \log(N)$ which is congruent to $a$ modulo $b$, or $N$ has a prime factor $p$ congruent to $a$ modulo $b$ in the interval
    \begin{align*}
        \left[ \lp \dfrac{a^2}{a^2-1} \dfrac{c_{k,\ell}^{a,b}(N)}{N^{k+\ell}} \rp^{-1/k}, \lp \dfrac{a^2}{a^2-1} \dfrac{c_{k,\ell}^{a,b}(N)}{N^{k+\ell}} \rp^{-1/k} \lp \dfrac{2a^2}{a^2-1} \omega_{a,b}(N) \rp^{1/k} \right],
    \end{align*}
    where $\omega_{a,b}(N)$ is the number of distinct primes $p \equiv a \pmod{b}$ that divide $N$.
\end{theorem}

\begin{remark}
    Since $c_{k,\ell}^{a,b}(N)=0$ unless $N$ possesses a prime factor congruent to $a$ modulo $b$, Theorem \ref{Thm: Quick Factorization Range Progressions} yields a nontrivial result only if $N$ possesses such a prime factor.
\end{remark}

\begin{theorem} \label{Thm: Primes in intervals Progressions}
    Assume that $1 \leq a < b$ are coprime integers, and assume there is at least one prime $p \equiv a \pmod{b}$. Then following are true:
    \begin{enumerate}
        \item Assume that $k>1$ is an odd integer and that for $c>1$ and $n\gg_c0$, every interval $(n,cn]$ contains at least one prime congruent to $a$ modulo $b$. Then for $\delta>0$ sufficiently small and $\varepsilon > 2\delta + 3\delta^2$, infinitely many of the values in $\{ c_{k,\ell}^{a,b}(N)/N^{k+\ell} \}_{N,\ell}$ intersect the interval $(\delta,\varepsilon)$.
        \item Assume that $k>1$ is an odd integer and that the sets
        \begin{align*}
            S_k = \bigg\{ \dfrac{c_{k,\ell}^{a,b}(N)}{N^{k+\ell}} : \ell>k \textrm{ odd, } \omega(N) \leq C_{a,b}(k) \bigg\},
        \end{align*}
        for some function $C_{a,b}(k)$ satisfying $C_{a,b}(k)^{1/k} \to 1$ as $k \to \infty$, intersect all intervals $(\delta,\varepsilon)$ with $\varepsilon > \delta + \delta^{(k+1)/k}$ and $\delta$ small.
        Then every interval $(n,cn]$ for $c>1$ and $n \gg_k 0$ contains at least one prime congruent to $a$ modulo $b$.
    \end{enumerate}
\end{theorem}

Since these methods describe properties of primes from the distributions of coefficients of quasimodular forms, which arise from results for the functions $\sigma_k(N)$, it would be interesting to understand if the phenomena described in this work have any implications beyond the study of prime numbers, e.g., the perfect number problem.

The paper proceeds as follows. In Section \ref{Sec: Prelims}, we lay out preliminary facts needed to prove the main results. In Section \ref{Sec: Key Lemma}, we prove Lemma \ref{Lem: Coefficient bounds}, the key lemma which acts as the engine behind the main theorems, and extensions of Lemma \ref{Lem: Coefficient bounds} to other prime-detecting quasimodular forms. In Sections \ref{Sec: Infinitude}--\ref{Sec: Intervals}, we prove Theorems \ref{Thm: Primes}, \ref{Thm: Quick Factorization Range}, and \ref{Thm: Primes in intervals}, respectively. In Section \ref{Sec: Progressions}, we prove Theorems \ref{Thm: Quick Factorization Range Progressions} and \ref{Thm: Primes in intervals Progressions}, and we discuss the possibility of proving results on Linnik's problem using this family of techniques. In Appendix \ref{Sec: Appendix A}, we discuss factorization of integers based on Theorem \ref{Thm: Quick Factorization Range} and from the partition-theoretic point of view. In Appendix \ref{Sec: Appendix B}, we illustrate an explicit examples of Theorem \ref{Thm: Quick Factorization Range}.

\section*{Acknowledgments}

The authors thank Alessandro Languasco, Pieter Moree and Robert Osburn for helpful commentary and for pointing out a flaw in one of our proofs in an earlier version of the manuscript. The views expressed in this article are those of the authors and do not reflect the official policy or position of the U.S. Naval Academy, Department of the Navy, the Department of War, or the U.S. Government. The first author was partially funded by the Irish Research Council Advanced Laureate Award IRCLA/2023/1934.

\section{Preliminaries} \label{Sec: Prelims}

\subsection{Quasimodular forms}

\begin{definition}
    Let $k \geq 2$ be an even integer. We define the Eisenstein series of weight $k$ by
    \begin{align*}
        G_k(\tau) = -\dfrac{B_k}{2k} + \sum_{n \geq 1} \sigma_{k-1}(n) q^n, \ \ \ \ \ q = e^{2\pi i \tau},
    \end{align*}
    where $\sigma_m(n) := \sum_{d|n} d^m$ and where $B_k$ are the Bernoulli numbers. We define a quasimodular form as any element of the algebra $\mathcal M := \C[G_2, G_4, G_6, G_8, \dots].$ We say that $f \in \mathcal M$ is of weight $k$ if $f$ is a sum of terms $c_j \prod_j G_j^{a_j}$ such that $\sum_j j a_j = k$.
\end{definition}

It is well known that $\mathcal M$ is closed under the differential operator $D = \dfrac{1}{2\pi i} \dfrac{d}{d\tau} = q \dfrac{d}{dq}.$ In particular, we note that $\mathcal M = \C[G_2, G_4, G_6]$, and it was shown by Ramanujan that
\begin{equation*}
    DG_2 = \frac{5}{6}G_2^2 - \frac{7}{3}G_4, \qquad DG_4 = 4G_2 G_4 - 14G_6, \text{ and} \qquad DG_6 = 6G_2 G_6 - \frac{20}{7}G_4^2.
\end{equation*}
See \cite{KanekoZagier} for more on the basic theory of quasimodular forms.

We will construct certain special elements of $\mathcal M$ as found in \cite{CraigIttersumOno}.

\begin{definition}
    For $\ell>k>0$ odd integers, define
    \begin{align*}
        f_{k,\ell}(\tau) := \lp D^\ell + 1 \rp G_{k+1}(\tau) - \lp D^k + 1 \rp G_{\ell+1}(\tau).
    \end{align*}
\end{definition}

We have the following result for $f_{k,\ell}(\tau)$. Recall Equation \eqref{fkl}.

\begin{lemma}[{\cite{CraigIttersumOno,Lelievre}}] \label{Lem: Asymptotic lemma}
    We have for $\ell>k>0$ odd integers that $f_{k,\ell} \in \mathcal M$. Moreover,
    \begin{align}\label{AsympLemEq}
       \sum_{n \geq 2} c_{k,\ell}(n) q^n = \sum_{n \geq 2} \lp \sum_{\substack{1<d<n \\ d|n}} \lp 1 + n^\ell \rp d^k - \lp 1 + n^k \rp d^\ell \rp q^n,
    \end{align}
    and for $n\geq 2, \ c_{k,\ell}(n) \geq 0$ with equality if and only if $n$ is prime.
\end{lemma}

In \cite[Thm. $11$]{CraigIttersumOno}, it was shown that the space of prime-detecting quasimodular Eisenstein series is spanned by the forms $D^nH_k$ for $n\geq 0, k\geq 6$ where 
\begin{equation}\label{HkEq}
    H_k:= \sum_{n\geq 0} b_n(H_k):= \begin{cases} \frac{1}{6} \lp (D^2-D+1)G_2 - G_4 \rp & k=6, \\ \frac{1}{24} \lp -D^2G_{k-6} + (D^2+1)G_{k-4} - G_{k-2}\rp & k\geq 8.\end{cases}
\end{equation}
In \cite{KaneKrishnamoorthyLau}, it was demonstrated using $L$-functions for $c_{k,\ell}(N)$ that this space is precisely the space of prime-detecting quasimodular forms, settling a conjecture from \cite{CraigIttersumOno}. See also \cite{IttersumMauthOnoSingh} for an alternative proof of this basis result using Galois representations.

\subsection{$L$-functions}
\begin{definition}
    Let $f(q)=\sum_{n\geq 0}a_n q^n$ be a $q$-series where $|q|<1.$ The $L$-function associated to $f$ is defined as
    \begin{equation*}
        L(f,q):= \sum_{n\geq 1} \frac{a_n}{n^s}
    \end{equation*}
    where $s\in\mathbb{C}$ such that $Re(s) \gg 0.$
\end{definition}

We can use \cite[Eq. 5.27]{DiamondShurman} to compute $L(G_{k+1},s)$, the $L$-function for the Eisenstein series of weight $k+1$ for an odd integer $k\geq 1.$ For the sake of keeping this paper as self-contained as possible, we will derive $L(G_{k+1},s)$ explicitly. 
\begin{lemma}\label{Lem: Eis Lfunc}
For Re$(s)>k+1,$
\begin{equation*}
    L(G_{k+1},s) = \zeta(s-k)\zeta(s).
\end{equation*}
    \begin{proof}
        \begin{align*}
        L(G_{k+1},s)&:= \sum_{n\geq 1} \frac{\sigma_{k}(n)}{n^s} =\sum_{n\geq 1}\frac{1}{n^s}\sum_{d|n}d^k\\
        &\ =\sum_{d\geq 1}d^k\sum_{m\geq 1}\frac{1}{(md)^s}=\sum_{d\geq 1}\frac{1}{d^{s-k}}\sum_{m\geq 1}\frac{1}{m^s}\\
        &\ =\zeta(s-k)\zeta(s).
    \end{align*}
    \end{proof}
    
\end{lemma}

\begin{lemma}\label{lem fkl Lfunc}
Let $\ell>k>0$ be odd integers, then for $Re(s) > \ell+k+1$ we have
\begin{equation*}
    L( f_{k,\ell},s) = (\zeta(s-k-\ell) - \zeta(s))(\zeta(s-\ell) - \zeta(s-k)).
\end{equation*}
\begin{proof}
Firstly, observe that if $f(q)=\sum_{n\geq 0}a_nq^n$ then $D^mf(q) = \sum_{n\geq 0}n^ma_nq^n$, so
\begin{equation*}
    L(D^mf,s) = \sum_{n\geq 1}\frac{n^ma_n}{n^s} = \sum_{n\geq 1}\frac{a_n}{n^{s-m}} = L(f,s-m),
\end{equation*}
as long as $s > m+1$. Therefore,
\begin{align*}
    L( f_{k,\ell},s)&:= L((1+D^\ell)G_{k+1} - (1+D^k)G_{\ell+1},s)\\
    &\ = L( G_{k+1},s ) + L( D^\ell G_{k+1},s ) - L ( G_{\ell+1} ,s) - L( D^kG_{\ell+1},s )\\
    &\ =\zeta(s-k)\zeta(s) + \zeta(s-k-\ell)\zeta(s-\ell) - \zeta(s-\ell)\zeta(s) - \zeta(s-\ell - k)\zeta(s-k)\\
    &\ =(\zeta(s-k-\ell) - \zeta(s))(\zeta(s-\ell) - \zeta(s-k)).
\end{align*}
\end{proof}
\end{lemma}

\begin{remark}
    Observe that Lemma \ref{lem fkl Lfunc} shows $L( f_{k,\ell},s)$ has the form $\lp\sum_{n \geq 1}a_nn^{-s}\rp \lp\sum_{n \geq 1}b_nn^{-s}\rp$. By Dirichlet convolution, we see that
    \begin{equation*}
        c_{k,\ell}(n) = \sum_{d|n} \lp \lp\frac{n}{d}\rp^{k+\ell} - 1\rp \lp d^\ell -  d^k\rp.
    \end{equation*} The prime-detecting property of $c_{k,\ell}(n)$ is immediate from this expression.
\end{remark}

\section{Behaviour of $c_{k,\ell}(n)$ with respect to the smallest prime divisor of $n$} \label{Sec: Key Lemma}

We now prove the following result which explains the source of the visible strips in Figure \ref{fig:c37}. This result is key to the proofs of Theorems \ref{Thm: Primes} and \ref{Thm: Quick Factorization Range}, and motivates the proofs for Theorems \ref{Thm: Quick Factorization Range Progressions} and \ref{Thm: Primes in intervals Progressions}.

\begin{lemma} \label{Lem: Coefficient bounds}
    Let $f_{k,\ell} = \sum_{n \geq 0} c_{k,\ell}(n) q^n$ for odd integers $\ell>k>1$. For fixed $n$, let $p$ be the smallest prime dividing $n$. Then we have for $n \geq 2$ that
    \begin{equation*}
        n^{\ell+k} p^{-k} \lp 1 - p^{k-\ell} \rp + n^k - n^\ell \leq c_{k,\ell}(n) < \lp 1 + n^\ell \rp n^k \sum_{\substack{d|n \\ p \leq d \leq \frac np}} \dfrac{1}{d^k}.
    \end{equation*}
\end{lemma}

\begin{remark}
    Using integration, one can show $c_{k,\ell}(n)< n^{\ell+k}p^{-k}\lp 1+ \frac{p}{k-1} \rp.$ Thus $c_{k,\ell}(n)/n^{k+\ell}$ is forced into arbitrarily narrow intervals for large values of $k$ and $\ell$.
\end{remark}

\begin{proof}
    If $n \not = p^2$, then by (\ref{AsympLemEq}) we have
    \begin{align*}
        c_{k,\ell}(n) \geq \sum_{d = p, \frac np} &\lp 1 + n^\ell \rp d^k - \lp 1 + n^k \rp d^\ell = \lp 1 + n^\ell \rp \lp p^k + \lp \dfrac{n}{p} \rp^k \rp - \lp 1 + n^k \rp \lp p^\ell + \lp \dfrac{n}{p} \rp\ell \rp 
        \\ &= n^{\ell+k}\lp \dfrac{1}{p^k} - \dfrac{1}{p^\ell} \rp + p^k - p^\ell + \lp \lp \dfrac{n}{p} \rp^\ell - \lp \dfrac{n}{p} \rp^k \rp \lp p^{\ell+k} - 1 \rp 
        \\ &\geq n^{\ell+k} p^{-k} \lp 1 - p^{k-\ell} \rp + p^k - p^\ell
        \ \geq \ n^{\ell+k} p^{-k} \lp 1 - p^{k-\ell} \rp + n^k - n^\ell.
    \end{align*}
    If $n = p^2$, then we have
    \begin{align*}
        c_{k,\ell}(n) = \lp 1 + p^{2\ell} \rp p^k - \lp 1 + p^{2k} \rp p^\ell = p^{2\ell+k} - p^{\ell + 2k} + p^k - p^\ell \geq p^{2\ell+2k} p^{-k} \lp 1 - p^{k-\ell} \rp + p^{2k} - p^{2\ell},
    \end{align*}
    which also proves the result.
    This proves the first inequality.
    
    For the upper bound, assuming now that $k>1$, if $p$ is the smallest prime divisor of $n$, then we have by flipping $d \to n/d$ that
    \begin{align*}
        c_{k,\ell}(n) = &\sum_{\substack{d|n \\ p \leq d \leq \frac np}} \lp 1 + n^\ell \rp d^k - \lp 1 + n^k \rp d^\ell
        = \sum_{\substack{d|n \\ p \leq d \leq \frac np}} \lp 1 + n^\ell \rp \dfrac{n^k}{d^k} - \lp 1 + n^k \rp \dfrac{n^\ell}{d^\ell}
        \\ \ &= \lp 1 + n^\ell \rp n^k \sum_{\substack{d|n \\ p \leq d \leq \frac np}} \dfrac{1}{d^k} - \lp 1 + n^k \rp n^\ell \sum_{\substack{d|n \\ p \leq d \leq \frac np}} \dfrac{1}{d^\ell} \ 
        < \ \lp 1 + n^\ell \rp n^k \sum_{\substack{d|n \\ p \leq d \leq \frac np}} \dfrac{1}{d^k},
    \end{align*}
    which proves the result.
\end{proof}

Our paper will focus its remaining effort on the $f_{k,\ell}$ forms specifically due to their particularly clean structure and coefficients. However, similar results could be derived using in place of $f_{k,\ell}$ the prime-detecting quasimodular forms $H_k$ of \cite{CraigIttersumOno} which, together with their derivatives, form a basis for the space of prime-detecting forms \cite{CraigIttersumOno,KaneKrishnamoorthyLau}. Recall (\ref{HkEq}).

\begin{lemma} \label{Lem: H_k coefficient bounds}
    Let $p$ be the smallest prime divisor of $n$, then
    \begin{equation*}
        n^3(p^{-1}-p^{-3}) \leq 6b_n(H_6) \leq n^3\left[ \sum_{p\leq d \leq \frac np}\frac  1d -\frac{1}{d^3}\right].
    \end{equation*}
    For $k\geq 8$,
    \begin{equation*}
        24b_n(H_k) \geq n^{k-3}\lp \frac{1}{p^{k-5}} - \frac{1}{p^{k-3}} \rp
    \end{equation*}
    and
    \begin{equation*}
        24b_n(H_K) \leq n^{k-3} \lp \frac{1}{p^{k-5}} -  \frac{1}{p^{k-3}} + \frac{1}{(4-k)p^{k-4}} - \frac{1}{(6-k)p^{k-6}}\rp + n^{k-5} \lp \frac{1}{p^{k-5}} - \frac{1}{(6-k)p^{k-6}}\rp.
    \end{equation*}
    
\end{lemma}

\begin{proof}
We first consider the case when $k=6$ and $n\geq 1$ with smallest prime divisor $p$ and $n \not = p^2$, i.e.,
\begin{align*}
    6b_n(H_6) = \sum_{d|n} &\lp n^2 - n + 1 \rp d - d^3=\lp n^2-n+1\rp\sum_{d|n}d - \sum_{d|n}d^3\\
    &\geq (n^2-n+1)\sum_{p,\frac np} d - \sum_{p,\frac np} d^3= \lp n^2-n+1\rp\left[p+\frac np\right]-\left[p^3 + \lp\frac np\rp^3\right].
\end{align*}
Grouping the terms with respect to powers of $n$ yields the first lower bound, and the case $n=p^2$ works similarly. We compute the upper bound as follows.
\begin{align*}
    6b_n(H_6)&= \lp n^2-n+1\rp\sum_{d|n}d - \sum_{d|n}d^3=\lp n^2-n+1\rp\sum_{d|n} \frac np - \sum_{d|n}\lp\frac np\rp^3\\
    &\qquad = n^3\left[ \sum_{d|n}\frac  1d -\lp\frac1d\rp^3\right] + \lp n-n^2\rp\sum_{d|n}\frac 1d\leq n^3\left[ \sum_{p\leq d \leq \frac np}\frac  1d -\frac{1}{d^3}\right].
\end{align*}
The proof is similar for $k \geq 8$.
\end{proof}

\section{Proof of Theorem \ref{Thm: Primes}} \label{Sec: Infinitude}

We now prove Theorem \ref{Thm: Primes} by constructing and analyzing the Dirichlet series for $f_{k,\ell}(\tau)$.

\begin{proof}
    Assume that there are finitely many primes, which we call $p_1 < p_2 < \cdots < p_r$. For $\mathcal P$ any subset of these primes and for $s>1,$ define
    \begin{equation*}
        \zeta(s; \mathcal P):=\sum_{\substack{n\geq 1\\(n,p)=1 \\ \text{ for all } p \in \mathcal P}} \frac{1}{n^s}.
    \end{equation*}
    This zeta function has the Euler product $\zeta(s;\mathcal P) = \prod\limits_{p\not\in \mathcal P}(1-p^{-s})^{-1}$. In particular, assuming that there are finitely many primes, observe that for $\mathcal P_r := \{ p_1, p_2, \dots, p_{r-1} \}$ and $P_r := p_1 p_2 \cdots p_{r-1}$ we have
    \begin{equation}\label{Eqn: ResZeta}
        \zeta(s;\mathcal P_r) = \sum\limits_{\substack{n\geq 1\\ (n,P_r)=1}}\frac{1}{n^s} = \frac{1}{1-p_r^{-s}}.
    \end{equation}
        
    In general, for a $q$-series $f(q) = \sum\limits_{n\geq 0} c(n)q^n$, and define a set of primes $\mathcal P$ the Dirichlet series
    $$L(f,s;P):= \sum\limits_{\substack{n\geq 1\\(n,p)=1\\\textrm{for all } p \in \mathcal P}} c(n)n^{-s}.$$ Now, by Lemma \ref{Lem: Coefficient bounds} and the fact that $x^{-k} - x^{-\ell}$ is a decreasing function for $\ell > k \geq 3$ odd and for $x \geq 2$, we have $c_{k,\ell}(n) \geq n^{k+\ell}\lp p^{-k} - p^{-\ell} \rp \geq n^{k+\ell}\lp p_r^{-k} - p_r^{-\ell} \rp$ for composite $n$. Then, for odd integers $\ell>k>0$, we have by this result and the vanishing of $c_{k,\ell}(n)$ at primes that
    \begin{align*}
        L(f_{k,\ell},s;\mathcal P_r) = \sum_{\substack{n\geq 1\\(n,P_r)=1}}\frac{c_{k,\ell}(n)}{n^s} &\geq \sum_{\substack{n\geq 1\\(n,P_r)=1}} \frac{(p_r^{-k} - p_r^{-\ell})n^{k+\ell}}{n^s} - (p_r^{-k}-p_r^{-\ell}) - \frac{(p_r^{-k} - p_r^{-\ell})p_r^{k+\ell}}{p_r^s}\\
        &=(p_r^{-k} - p_r^{-\ell})\lp\zeta(s-k-\ell;\mathcal P_r) - 1 - p_r^{k+\ell-s}\rp,
    \end{align*}
    and by \eqref{Eqn: ResZeta} it follows that
    \begin{align*}
        L(f_{k,\ell},s;\mathcal P_r) = (p_r^{-k} - p_r^{-\ell})\lp \frac{1}{1-p_r^{k+\ell-s}}  - 1 - p_r^{k+\ell-s}\rp = \frac{p_r^{2(k+\ell-s)}(p_r^{-k} - p_r^{-\ell})}{1-p_r^{k+\ell-s}}
    \end{align*}

    On the other hand, by Lemma \ref{lem fkl Lfunc} and \eqref{Eqn: ResZeta}, we have
    \begin{equation*}
        L(f_{k,\ell},s;\mathcal P_r) = \lp\frac{1}{1-p_r^{k+\ell-s}} - \frac{1}{1-p_r^{-s}}\rp\lp \frac{1}{1-p_r^{\ell-s}} - \frac{1}{1-p_r^{k-s}} \rp.
    \end{equation*}
    Therefore, for all $s>k+\ell+1$, we must have
    \begin{equation*}
        \lp\frac{1}{1-p_r^{k+\ell-s}} - \frac{1}{1-p_r^{-s}}\rp\lp \frac{1}{1-p_r^{\ell-s}} - \frac{1}{1-p_r^{k-s}} \rp \geq \frac{p_r^{2(k+\ell-s)}(p_r^{-k} - p^{-\ell})}{1-p_r^{k+\ell-s}}.
    \end{equation*}
    Simplifying this expression, we obtain
    \begin{align*}
        p_r^{-k} - p_r^{-\ell} &\leq \frac{1-p_r^{k+\ell-s}}{p_r^{2(k+\ell-s)}}\cdot\frac{p_r^{-s}(p_r^{k+\ell}-1)}{(1-p_r^{k+\ell-s})(1-p_r^{-s})}\cdot\frac{p_r^{-s}(p_r^\ell-p_r^k)}{(1-p_r^{\ell-s})(1-p_r^{k-s})}\\
        &=\frac{(p_r^{k+\ell}-1)(p_r^\ell - p_r^k)}{p_r^{2(k+\ell)}(1-p_r^{-s})(1-p_r^{\ell-s})(1-p_r^{k-s})}.
    \end{align*}
    Since this inequality holds true for every $s > \ell + k + 1$, and the left-hand side is independent of $s$, we may take $s \to \infty$ and obtain the inequality
    \begin{equation*}
        p_r^{-k} - p_r^{-\ell} \leq \frac{(p_r^{k+\ell}-1)(p_r^\ell - p_r^k)}{p_r^{2(k+\ell)}} = (1-p_r^{-(k+\ell)})(p_r^{-k} - p_r^{-\ell}) < p_r^{-k} - p_r^{-\ell},
    \end{equation*}
    a contradiction. Therefore, there must be infinitely many primes.
\end{proof}

\section{Proof of Theorem \ref{Thm: Quick Factorization Range}}

Let $\ell > k \geq 3$ be odd integers and $N \geq 3$ be a composite integer with minimal prime divisor $p \geq \log(N)$. Consider the normalized coefficients $\theta_{k,\ell}(N) := \frac{c_{k,\ell}(N)}{N^{\ell+k}}$. Then by Lemma \ref{Lem: Coefficient bounds}, we have
\begin{align*}
    \dfrac{1}{p^k} - \dfrac{1}{p^\ell} < \theta_{k,\ell}(N) < \sum_{\substack{d|N \\ p \leq d \leq N/p}} \dfrac{1}{d^k}.
\end{align*}
Note that since $p \geq 2$ and $k \geq 3$, we can verify with classical integral bounds that $\theta_{k,\ell}(n) < 1$.

For $N \gg 0$, we have
\begin{align*}
    \sum_{\substack{d|N \\ p \leq d \leq N/p}} \dfrac{1}{d^k} < \lp \dfrac{1}{1 - p^{-k}} \rp^{\omega(N)} - 1, 
\end{align*}
Since $\frac{1}{1-x} \leq e^{x+x^2} \leq e^{9x/8}$ for $x \leq \frac 18$, we may apply this inequality in the case $x = p^{-k}$ and obtain
\begin{align*}
    \sum_{\substack{d|N \\ p \leq d \leq N/p}} \dfrac{1}{d^k} < e^{9\omega(N)/8p^k} - 1 < \dfrac{2\omega(N)}{p^k}.
\end{align*}
since $e^{9x/8} - 1 < 2x$ when $0 \leq x \leq \frac 12$, where the final inequality follows from $\omega(N) < \log(N) \leq p$. Thus, for $N \gg 0$,
\begin{align*}
    \sum_{\substack{d|N \\ p \leq d \leq N/p}} \dfrac{1}{d^k} < \dfrac{2 \omega(N)}{p^k}.
\end{align*}
Furthermore, since $p^{-k} - p^{-\ell} \geq \frac 34 p^{-k}$ and $N \gg 0$, we have 
\begin{align} \label{Eqn: Wide interval eqn}
    \dfrac{3}{4p^k} < \theta_{k,\ell}(N) < \dfrac{2 \omega(N)}{p^k}.
\end{align}
Since $\frac 34 x^{-k}$ is positive and decreasing for $x > 1$ and $k$ fixed, we can identify the unique minimal integer $x_N$ such that
\begin{align*}
    \dfrac{3}{4 x_N^k} < \theta_{k,\ell}(N).
\end{align*}
Note that we then have $p \geq x_N$.

If for $N \gg 0$ we had the simultaneous inequalities $p > x_N + \delta_N$ for some $\delta_N>0$ and
\begin{align} \label{Eqn: delta xn ineq}
    \dfrac{3}{4 x_N^k} \geq \dfrac{2 \omega(N)}{\lp x_N + \delta_N \rp^k},
\end{align}
we obtain a contradiction to \eqref{Eqn: Wide interval eqn}. Thus, if we choose some $\delta_N$ satisfying \eqref{Eqn: delta xn ineq}, then we must have $x_N \leq p < x_N + \delta_N$ for $N \gg 0$. By \eqref{Eqn: delta xn ineq}, we may choose any
\begin{align*}
    \delta_N \geq \lp \lp \dfrac{8}{3} \omega(N) \rp^{1/k} - 1 \rp x_N.
\end{align*}
Therefore by choosing equality above, $N$ must have a prime divisor in the interval
\begin{align*}
     \left[ x_N, \lp \dfrac{8}{3} \omega(N) \rp^{1/k} x_N \right].
\end{align*}
Finally, from the inequalities $\frac 34 x^{-k}_N < c_{k,\ell}(N)/N^{k+\ell} \leq \frac 34 \lp x_N - 1 \rp^{-k}$, which hold by the definition of $x_N$, we have
\begin{align*}
    \lp \dfrac 43 \dfrac{c_{k,\ell}(N)}{N^{k+\ell}} \rp^{-1/k} < x_N \leq \lp \dfrac 43 \dfrac{c_{k,\ell}(N)}{N^{k+\ell}} \rp^{-1/k} + 1,
\end{align*}
which completes the proof.

\section{Proof of Theorem \ref{Thm: Primes in intervals}} \label{Sec: Intervals}

For any $N \geq 1$ and $\ell > k \geq 3$, we have
\begin{align} \label{Eqn: Norm ckl}
    \dfrac{c_{k,\ell}(N)}{N^{\ell+k}} = \lp 1 + \dfrac{1}{N^\ell} \rp \dfrac{\sigma_k(N)}{N^k} - \lp 1 + \dfrac{1}{N^k} \rp \dfrac{\sigma_\ell(N)}{N^\ell}.
\end{align}

To prove Theorem \ref{Thm: Primes in intervals}, we will demonstrate that values of \eqref{Eqn: Norm ckl}, as $\ell$ and $N$ vary simultaneously, intersect certain shrinking intervals to the right of zero. We then use Theorem \ref{Thm: Quick Factorization Range} to guarantee a prime in every interval $(x,cx)$ for $c>1$ and $x\gg_c 0$. We note that efforts to extend our results to primes in short intervals appear to require improvements to Theorem \ref{Thm: Quick Factorization Range}.

The proof of Theorem \ref{Thm: Primes in intervals} requires the key density calculation.

\begin{proposition} \label{Prop: Density}
    Assume that for $c>1$ and $n\gg_c0$, the interval $(n,cn]$ contains at least one prime. Let $k>1$ be an odd integer and let $\delta > 0$ and $\varepsilon > 2\delta + 3\delta^2$ be positive real numbers. Then, if $\delta$ is sufficiently small as a function of $k$, there are infinitely many integers $M\gg0$ such that
    \begin{align*}
        1+\delta < \dfrac{\sigma_k(M)}{M^k} < 1 + \varepsilon.
    \end{align*}
\end{proposition}

\begin{proof}[Proof of Proposition \ref{Prop: Density}]
    For $\delta$ sufficiently small, by the hypothesis we can identify infinitely many distinct primes $\{ p_j \}_{j \geq 0}$ satisfying $2^{\frac{j}{k}}/\delta^{1/k} \leq p_j \leq 2^{\frac{j+1}{k}} / \delta^{1/k}$, alternatively
    \begin{align} \label{Eqn: pj Ineq}
        1 + \dfrac{\delta}{2^{j+1}} \leq 1 + \dfrac{1}{p_j^k} \leq 1 + \dfrac{\delta}{2^j}.
    \end{align}
    Now, consider the sequence of real numbers
    \begin{align*}
        A_N := \dfrac{\sigma_k(p_1 \cdots p_N)}{(p_1\cdots p_N)^k} = \prod_{j=0}^N \lp 1 + \dfrac{1}{p_j^k} \rp;
    \end{align*}
    by \eqref{Eqn: pj Ineq},
    \begin{align*}
        \prod_{j=1}^{N+1} \lp 1 + \dfrac{\delta}{2^j} \rp \leq A_N \leq \prod_{j=0}^N \lp 1 + \dfrac{\delta}{2^j} \rp.
    \end{align*}
    Now, if we formally take $N \to \infty$, we obtain
    \begin{align*}
        \Log\lp \prod_{j=1}^\infty \lp 1 + \dfrac{\delta}{2^j} \rp \rp = \sum_{j \geq 1} \Log\lp 1 + \dfrac{\delta}{2^j} \rp \geq \sum_{j \geq 1} \lp \dfrac{\delta}{2^j} - \dfrac{\delta^2}{2^{2j+1}} \rp = \delta - \dfrac{\delta^2}{6},
    \end{align*}
    from which we obtain
    \begin{align*}
        \lim_{N \to \infty} A_N \geq e^{\delta - \delta^2/6}.
    \end{align*}
    From a Taylor expansion $e^{\delta - \delta^2/6} = 1 + \delta + \frac 13 \delta^2 + O(\delta^3)$, we then deduce that for $\delta$ sufficiently small, there is some minimal choice of $N$ satisfying $A_N > 1+\delta$; we fix this value of $N$.
    
    Now, for positive integers $m_0, m_1, \dots, m_N \geq 1$, we consider the set of products
    \begin{align*}
        A_N(\underline{m}) := \prod_{j=0}^N \left[ 1 + p_j^{-k} + p_j^{-2k} + \dots + p_j^{-m_j k} \right]
    \end{align*}
    where $\underline{m} = (m_0,m_1,\dots,m_N).$ We will then show that there are infinitely many possible $\underline{m}$ such that $1+\delta < A_N(\underline{m}) < 1+\varepsilon$. Making use of the limitation $p_j \geq 2^{\frac{j}{k}} / \delta^{1/k}$, we see that
    \begin{align*}
        \prod_{j=0}^N \left[ 1 + p_j^{-k} + p_j^{-2k} + \dots \right] \leq \prod_{j=0}^N \dfrac{1}{1 - \frac{\delta}{2^{j}}} =: P_N(\delta).
    \end{align*}
    For our choice of $N$ and $\delta > 0$, we have
    \begin{align} \label{Eqn: Pinf Ineq}
        1+\delta < A_N(\underline{m}) < P_N(\delta) < P_\infty(\delta) := 1 + \sum_{m \geq 1} \dfrac{2^{m(m+1)/2}}{(2-1)(4-1) \cdots (2^m-1)} \delta^m.
    \end{align}
   By computing the coefficients of $\delta^m$ for $m=1,2$ in the right hand side of \eqref{Eqn: Pinf Ineq}, one can verify
    \begin{align*}
       P_\infty(\delta) < 1 + 2\delta + 3\delta^2 < 1 + \varepsilon
    \end{align*}
    since $\delta$ is chosen to be sufficiently small. Note that $A_N(\underline{m}) \to P_N(\delta)$ as $m_j \to \infty$ for each $j$ simultaneously. Then, using the fact that $1+\delta < P_N(\delta) < 1+\varepsilon$ for our choice of $N$ and $\delta$, we have $1+\delta < A_N(\underline{m}) < 1+\varepsilon$ when $m_j \gg 0$ for each $j$, in particular, there are infinitely many permissible choices for the $m_j$.
    The result then follows since
    \begin{align*}
        A_N(\underline{m}) = \dfrac{\sigma_k\lp p_0^{m_0} p_1^{m_1} \cdots p_N^{m_N} \rp}{\lp p_0^{m_0} p_1^{m_1} \cdots p_N^{m_N} \rp^k}.
    \end{align*}
\end{proof}

We then prove the first part of Theorem \ref{Thm: Primes in intervals} as Corollary \ref{Cor: Intervals Part 1}

\begin{corollary} \label{Cor: Intervals Part 1}
    Assume that $k>1$ is an odd integer and that for $c>1$ and $n\gg_c0$, every interval $(n,cn]$ contains at least one prime. Then for $\delta>0$ sufficiently small and $\varepsilon > 2\delta + 3\delta^2$, there are infinitely many values in the set $\{ c_{k,\ell}(N)/N^{k+\ell} \}_{N,\ell}$ that intersect the interval $(\delta,\varepsilon)$.
\end{corollary}

\begin{proof}
    Let $\delta>0$ and $\varepsilon > 2\delta + 3\delta^2$ be real numbers, with $\delta$ sufficiently small according to the requirements of Proposition \ref{Prop: Density}. It suffices to show that $\frac{c_{k,\ell}(N)}{N^{k+\ell}} \in (\delta,\varepsilon)$ for some $N$. Assuming that every interval $(n,2^{1/k}n]$ for $n\gg_k0$ contains at least one prime, Proposition \ref{Prop: Density} implies that there are infinitely many $N$ satisfying
    \begin{align*}
        1+\delta < \dfrac{\sigma_k(N)}{N^k} < 1+\varepsilon.
    \end{align*}
    Observe that $\sigma_\ell(N)/N^\ell \to 1$ as $\ell\to\infty$, so for arbitrarily large values of $N$ we may choose $\ell > k$ satisfying
    \begin{align*}
        \delta < \dfrac{\sigma_k(N)}{N^k} - \dfrac{\sigma_\ell(N)}{N^\ell} < \varepsilon.
    \end{align*}
    Since $N$ may be taken as large as desired, and since
    \begin{align*}
        \left| \dfrac{\sigma_k(N)}{N^k} - \dfrac{\sigma_\ell(N)}{N^\ell} - \dfrac{c_{k,\ell}(N)}{N^{k+\ell}} \right| \leq \dfrac{\sigma_k(N)}{N^{k+\ell}} + \dfrac{\sigma_\ell(N)}{N^{k+\ell}} < \dfrac{2\zeta(k)}{N^k},
    \end{align*}
    we have the convergence
    \begin{align*}
        \dfrac{\sigma_k(N)}{N^k} - \dfrac{\sigma_\ell(N)}{N^\ell} \to \dfrac{c_{k,\ell}(N)}{N^{k+\ell}}
    \end{align*}
    as $N \to \infty$ with error independent of $\ell$. In particular, with $k$ fixed, we can choose infinitely many $N \gg 0$ so that $\delta < c_{k,\ell}(N)/N^{k+\ell} < \varepsilon$, which completes the proof.
\end{proof}

We now prove the second part of Theorem \ref{Thm: Primes in intervals} as Corollary \ref{Cor: Intervals Part 2}.

\begin{corollary} \label{Cor: Intervals Part 2}
    Assume that $k>1$ is an odd integer and that the sets
    \begin{align*}
        S_k = \bigg\{ \dfrac{c_{k,\ell}(N)}{N^{k+\ell}} : \ell>k \textrm{ odd, } \omega(N) \leq C(k) \bigg\},
    \end{align*}
    for some function $C(k)$ satisfying $C(k)^{1/k} \to 1$ as $k \to \infty$, intersect all intervals $(\delta,\varepsilon)$ with $\varepsilon > \delta + \delta^{(k+1)/k}$ and $\delta$ small.
    Then every interval $(n,cn]$ for $c>1$ and $n \gg_k 0$ contains at least one prime.
\end{corollary}

\begin{proof}
    Recall from \eqref{Eqn: Wide interval eqn} that either $p < \log(N)$ or
    \begin{align*}
        \dfrac{3}{4p^k} < \dfrac{c_{k,\ell}(N)}{N^{\ell+k}} < \dfrac{2 \omega(N)}{p^k}.
    \end{align*}
    where $p$ is the minimal prime dividing $N$. Following the same procedure as implemented in the proof of Theorem \ref{Thm: Quick Factorization Range} immediately following \eqref{Eqn: Wide interval eqn}, if we are able to identify $M$ as the unique integer such that
    \begin{align} \label{Eqn: Interval}
        \dfrac{3}{4M^k} < \dfrac{c_{k,\ell}(N)}{N^{\ell+k}} < \dfrac{3}{4\lp M - 1 \rp^k},
    \end{align}
    we obtain the existence of a prime $p|N$ satisfying
    \begin{align} \label{Eqn: Prime}
        M < p < \lp \dfrac{8\omega(N)}{3} \rp^{1/k} M.
    \end{align}
    One can verify that if $\delta = \frac{3}{4M^k}$ and $\varepsilon = \frac{3}{4 (M-1)^k}$ for $M\in\mathbb{Z}^{>0}$, then $\delta + \delta^{(k+1)/k} < \varepsilon$. Observe that by the hypothesis, for $M\gg0$ we can find some $N\gg0$ satisfying \eqref{Eqn: Interval} and, in turn, some prime $p$ satisfying \eqref{Eqn: Prime}. Since $C(k)^{1/k} \to 1$ as $k \to \infty$, this completes the proof.
\end{proof}

\section{Primes in Arithmetic Progressions} \label{Sec: Progressions}

In this section, we will prove analogs of Theorems \ref{Thm: Quick Factorization Range} and \ref{Thm: Primes in intervals} for primes in arithmetic progressions. Furthermore, we briefly discuss how our techniques may be applied to the problem of bounding the first prime in an arithmetic progression.

\subsection{Preliminary Bounds}

To begin, we need an analogy of Lemma \ref{Lem: Coefficient bounds}.

\begin{lemma} \label{Lem: Coefficient bounds progressions}
    Let $1 \leq a \leq b$ be coprime integers and let $n \geq 1$. Then we have $c_{k,\ell}^{a,b}(n) = 0$ if and only if $n \equiv a \pmod{b}$ is prime, and for all $n$ we have
    \begin{align*}
        n^{\ell+k} \lp \dfrac{1}{p^k} - \dfrac{1}{p^\ell} \rp < c_{k,\ell}^{a,b}(n) < n^{\ell+k} \sum_{\substack{d|n \\ d \in \mathcal N_{a,b}}} \dfrac{1}{d^k},
    \end{align*}
    where $p$ is the smallest prime satisfying $p \equiv a \mod{b}$ and $p|n$.
\end{lemma}

\begin{proof}
    Observe firstly that $c_{k,\ell}^{a,b}(p) = 0$ if $p \equiv a \pmod{b}$ is prime; so assume now that $n$ is not such a $p$.
    If the only divisor of $n$ in $\mathcal N_{a,b}$ is 1, then $c_{k,\ell}^{a,b}(n) = (1+n^\ell) - (1+n^k) = n^\ell - n^k > 0$, and so we assume now that $n$ possesses nontrivial divisors in $\mathcal N_{a,b}$; letting $d$ be such a divisor,
    \begin{align} \label{Eqn: Starting Lower Bound}
        \lp 1 + n^\ell \rp d^k - \lp 1 + n^k \rp d^\ell &= d^k \left[ n^\ell + 1 - d^{\ell-k} - n^k d^{\ell-k} \right] \geq d^k \left[ n^\ell + 1 - \dfrac{n^{\ell-k}}{2^{\ell-k}} - \dfrac{n^{\ell}}{2^{\ell-k}} \right]
        \notag \\ &= d^k \left[ n^\ell \lp 1 - \dfrac{1}{2^{\ell-k}} - \dfrac{1}{n^k 2^{\ell-k}} \rp + 1 \right] > 0.
    \end{align}
    This establishes the positivity and vanishing criteria claimed.

    Now, let $n \geq 1$ be composite and assume $n$ has nontrivial divisors in $\mathcal N_{a,b}$. Then we have
    \begin{align*}
        c_{k,\ell}^{a,b}(n) &= \sum_{\substack{d|n \\ d \in \mathcal N_{a,b}}} \lp 1 + n^\ell \rp d^k - \lp 1 + n^k \rp d^\ell 
        = \sum_{\substack{d|n \\ \substack{\frac nd \in \mathcal N_{a,b}}}} \lp 1 + n^\ell \rp \dfrac{n^k}{d^k} - \lp 1 + n^k \rp \dfrac{n^\ell}{d^\ell}
        \\ &= \sum_{\substack{d|n \\ \frac nd \in \mathcal N_{a,b}}} n^{\ell+k} \lp \dfrac{1}{d^k} - \dfrac{1}{d^\ell} \rp + \lp \dfrac{n^k}{d^k} - \dfrac{n^\ell}{d^\ell} \rp
        < \sum_{\substack{d|n \\ \frac nd \in \mathcal N_{a,b}}} \dfrac{n^{\ell+k}}{d^k} = n^\ell \sum_{\substack{d|n \\ d \in \mathcal N_{a,b}}} d^k = n^{\ell+k} \sum_{\substack{d|n \\ d \in \mathcal N_{a,b}}} \dfrac{1}{d^k}.
    \end{align*}
    We also obtain, for $p$ the smallest prime congruent to $a$ modulo $b$ that divides $n$,
    \begin{align*}
        c_{k,\ell}^{a,b}(n) = \sum_{\substack{d|n \\ \frac nd \in \mathcal N_{a,b}}} n^{\ell+k} \lp \dfrac{1}{d^k} - \dfrac{1}{d^\ell} \rp + \lp \dfrac{n^k}{d^k} - \dfrac{n^\ell}{d^\ell} \rp > n^{\ell+k} \sum_{\substack{d|n \\ \frac nd \in \mathcal N_{a,b}}} \lp \dfrac{1}{d^k} - \dfrac{1}{d^\ell} \rp > n^{\ell+k} \lp \dfrac{1}{p^k} - \dfrac{1}{p^\ell} \rp.
    \end{align*}
    This completes the proof.
\end{proof}

\subsection{Proof of Theorem \ref{Thm: Quick Factorization Range Progressions}}

Let $N \geq 3$ be composite, $\ell > k \geq 3$ be odd integers and let $a,b$ be coprime integers with $2 \leq a \leq b+1$. Suppose $N$ has minimal prime divisor $p \equiv a \pmod{b}$ and that $p \geq \log(N)$. Consider  $\theta_{k,\ell}^{a,b}(N) := \frac{c_{k,\ell}^{a,b}(N)}{N^{\ell+k}}$.

Then by Lemma \ref{Lem: Coefficient bounds}, we have
\begin{align*}
    \dfrac{1}{p^k} - \dfrac{1}{p^\ell} < \theta_{k,\ell}^{a,b}(N) < \sum_{\substack{d|N \\ d \in \mathcal N_{a,b} \\ p \leq d \leq N/p}} \dfrac{1}{d^k}.
\end{align*}
It is easy to verify by the same method as in the proof of Theorem \ref{Thm: Quick Factorization Range} that $\theta_{k,\ell}^{a,b}(N) < 1$ and that for $N \gg 0$ we have
\begin{align*}
    \sum_{\substack{d|N\\ d \in \mathcal N_{a,b}  \\ p \leq d \leq N/p}} \dfrac{1}{d^k} < \dfrac{2\omega_{a,b}(N)}{p^k},
\end{align*}
where in the last step we have used the inequality $\omega_{a,b}(N) < \log(N) \leq p$ and $e^x-1 < 2x$ for $x<1$. Since we also have for $\ell > k \geq 3$ odd that $p^{-k} - p^{-\ell} \geq p^{-k} \lp 1 - p^{-2} \rp \geq \frac{a^2-1}{a^2} p^{-k}$, we have for $N \gg 0$ that
\begin{align} \label{Eqn: Wide interval eqn progs}
    \dfrac{a^2-1}{a^2 p^k} < \theta_{k,\ell}^{a,b}(N) < \dfrac{2 \omega_{a,b}(N)}{p^k}.
\end{align}
Find the unique minimal integer $x_{N,a,b}$ such that
\begin{align*}
    \dfrac{a^2-1}{a^2 x_{N,a,b}^k} < \theta_{k,\ell}^{a,b}(N),
\end{align*}
so that $p \geq x_{N,a,b}$.  Now, if $p > x_{N,a,b} + \delta_{N,a,b}$ for some $\delta_{N,a,b}>0$ and simultaneously
\begin{align*}
    \dfrac{a^2-1}{a^2 x_{N,a,b}^k} \geq \dfrac{2 \omega_{a,b}(N)}{\lp x_{N,a,b} + \delta_{N,a,b} \rp^k}
\end{align*}
for $N \gg 0$, then we have contradicted \eqref{Eqn: Wide interval eqn progs}, thus if we choose such a $\delta_{N,a,b}$, then we must have $x_{N,a,b} \leq p < x_{N,a,b} + \delta_{N,a,b}$ for $N \gg 0$. We find that any choice
\begin{align*}
    \delta_{N,a,b} \geq \lp \lp \dfrac{2a^2}{a^2-1} \omega_{a,b}(N) \rp^{1/k} - 1 \rp x_{N,a,b}
\end{align*}
creates the above contradiction, and therefore $N$ must have a prime divisor in the interval
\begin{align*}
    \left[ x_{N,a,b}, \lp \dfrac{2a^2}{a^2-1} \omega_{a,b}(N) \rp^{1/k} x_{N,a,b} \right].
\end{align*}
Finally, from the inequalities $\frac{a^2-1}{a^2} x^{-k}_{N,a,b} < c_{k,\ell}^{a,b}(N)/N^{k+\ell} \leq \frac{a^2-1}{a^2} \lp x_{N,a,b} - 1 \rp^{-k}$, which hold by the definition of $x_{N,a,b}$, we have
\begin{align*}
    \lp \dfrac{a^2}{a^2-1} \dfrac{c_{k,\ell}^{a,b}(N)}{N^{k+\ell}} \rp^{-1/k} < x_{N,a,b} \leq \lp \dfrac{a^2}{a^2-1} \dfrac{c_{k,\ell}^{a,b}(N)}{N^{k+\ell}} \rp^{-1/k} + 1,
\end{align*}
which completes the proof.

\subsection{Proof of Theorem \ref{Thm: Primes in intervals Progressions}}

The proof of Theorem \ref{Thm: Primes in intervals Progressions} follows {\it mutatis mutandis} from the density proposition analogous to Proposition \ref{Prop: Density}.

\begin{proposition} \label{Prop: Density Progressions}
    Assume that for $c>1$ and $n\gg_c0$, the interval $(n,cn]$ contains at least one prime congruent to $a$ modulo $b$. Let $k>1$ be an odd integer and let $\delta > 0$ and $\varepsilon > 2\delta + 3\delta^2$ be positive real numbers. Then, if $\delta$ is sufficiently small as a function of $k$, there are infinitely many integers $M\gg0$ such that
    \begin{align*}
        1+\delta < \dfrac{\sigma^{a,b}_k(M)}{M^k} < 1 + \varepsilon.
    \end{align*}
\end{proposition}

\begin{proof}[Proof of Proposition \ref{Prop: Density Progressions}]
    Since Theorem \ref{Thm: Primes in intervals Progressions} gives us the existence of infinitely many such primes under only the given assumption, the result follows {\it mutatis mutandis} from the proof of Proposition \ref{Prop: Density}.
\end{proof}

\subsection{Smallest prime in an arithmetic progression}

Our results on primes in arithmetic progressions rely, in an indispensable way, on the assumption that such a prime does in fact exist. Consequently, we leave a note on the possibility to establish the existence of such a prime by using these techniques directly. Note that typical ways of doing this, like Dirichlet's classic proof with $L$-functions, already prove in a single stroke the infinitude of primes.

While we do not claim to have such a proof, there are hints from our results covering all primes that may apply in this setting. In particular, the prime-detecting forms $f_{k,\ell}$ has a Fourier expansion approximated from below by
\begin{align*}
    \sum_{p \text{ prime}} \sum_{\substack{n \geq 1 \\ P_{\text{min}}(n) = p}}n^{k+\ell}p^{-k}q^n,
\end{align*}
where $P_{\text{min}}(n)$ denotes the minimal prime factor of $n$. The inner sums are very similar to the functions $L_{m,k,\ell}(q) := m^{-k} \sum_{n \geq 1} n^{k+\ell} q^n$, which are essentially polylogarithms, and so are suitable for asymptotic analysis. Now, if we consider
\begin{align*}
    f_{k,\ell}(q) - \sum_{m \geq 0} L_{a+mb,k,\ell}(q),
\end{align*}
the existence of some prime $p = a+mb$ should produce some cancellation in the main term asymptotics as $q \to 1$. Plausibly, such results if carefully collected may not only find that there must be some prime $p \equiv a \pmod{b}$, but even in principle might give bounds on when the first such prime might occur. We leave such speculation as a possible future topic of study.

\appendix
\section{Factorization in light of Theorem \ref{Thm: Quick Factorization Range}} \label{Sec: Appendix A}

Based on \cite{CraigIttersumOno}, Theorem \ref{Thm: Quick Factorization Range} provides a partition-theoretic algorithm for the factorization of $N$. This method calculates the coefficients of $f_{k,\ell}$ using MacMahon's functions after which we identify potential ranges of values in which the smallest prime dividing $N$ might lie. Before we discuss the computational complexity issues surrounding this algorithm, it must be noted that the calculation of coefficients of depth one MacMahon functions $M_{(a)}(n)$ for any $a \geq 0$ involves calculating every partition $(d,d,\dots,d)$ of $n$, which is the factorization problem which we claim to solve. Therefore, a proper partition-theoretic algorithm for factoring $N$ must avoid using any depth one MacMahon forms. These formulas can be found, for example, using simple linear algebra; one such example is given by
\begin{align*}
    H_8(q) - \frac{31}{120960} = \frac{8}{9}\mathcal{U}_{(1,1)} &- \frac{395}{324}\mathcal{U}_{(3,1)} - \frac{41}{36}\mathcal{U}_{(1,3)} + \frac{13}{54}\mathcal{U}_{(2,2)} - 2\mathcal{U}_{(1,1,1)} - \frac{151}{324}\mathcal{U}_{(5,1)}-\frac{1}{4}\mathcal{U}_{(1,5)} \\&+ \frac{70}{81}\mathcal{U}_{(4,2)} + \frac{175}{162}\mathcal{U}_{(2,4)} + \frac{20}{3}\mathcal{U}_{(1,1,3)} + \frac{20}{3}\mathcal{U}_{(1,3,1)} + \frac{20}{3}\mathcal{U}_{(3,1,1)}.
\end{align*}
Note that the coefficients of $H_k(q)$ are defined a priori from powers of $n$ times divisor sum functions, and from this point of view would appear to be depth one objects. However, with careful use of the Bachmann--K\"{u}hn quasi-shuffle relations for $q$-multiple zeta values, such identities can be derived for the vast majority of (if not all) prime-detecting quasimodular forms. Furthermore, Theorem \ref{Thm: Quick Factorization Range} generalizes readily to take other prime-detecting forms as inputs once the analog of Lemma \ref{Lem: Coefficient bounds} is established, which for $H_k$ is precisely Lemma \ref{Lem: H_k coefficient bounds}. We may therefore remain unconcerned about issues of circularity in the proposed factorization procedure.

Since Theorem \ref{Thm: Quick Factorization Range} outlines a particular interval of integers within which $N$ must have a prime factor, it is natural to investigate if this is a useful factorization algorithm. While the answer is essentially no due to the difficulty of enumerating partitions, there are nonetheless suggestive questions which should be explored arising from the fact that the intervals in Theorem \ref{Thm: Quick Factorization Range} don't fundamentally rely on the exact values of these coefficients, but on their orders of magnitude. Thus, we will explore the ways in which hypothetical approximations of $c_{k,\ell}(N)$ still inform the factorization of $N$.

Assume for the remainder of discussion that $N$ has no prime divisor less than $\log(N)$, which can be verified in polynomial time. We know from \eqref{Eqn: Wide interval eqn} that $\frac 34 p^{-k} < c_{k,\ell}(N)/N^{k+\ell}$, and from this we may deduce that
\begin{align*}
    p > \lp \dfrac{3 N^{k+\ell}}{4 c_{k,\ell}(N)} \rp^{1/k}.
\end{align*}
We can then choose a particular function $\delta_N>0$ so that we are guaranteed $p \geq N^\delta$, namely,
\begin{align*}
    \delta = \dfrac{1}{k \log(N)} \log\lp \dfrac{3 N^{k+\ell}}{4 c_{k,\ell}(N)} \rp.
\end{align*}
Thus, the value of $c_{k,\ell}(N)$ is directly suggestive of whether $N$ possesses only large prime divisors.

In order to permit some interesting flexibility and a possible computational speed-up to the algorithm, we might consider using approximations of $c_{k,\ell}(N)$ in place of its exact value. To this end, let $a_{k,\ell}(N)$ be such a hypothetical approximation\footnote{We do not take too much care over how close the approximation is, as our discussion here is primarily philosophical in nature.}, and subsequently we might consider a value
\begin{align*}
    \delta = \dfrac{1}{k \log(N)} \log\lp \dfrac{3 N^{k+\ell}}{4 a_{k,\ell}(N)} \rp.
\end{align*}
If $\delta > \frac 12$, this is aberrant and suggests that $N$ is probably semiprime with its two primes extremely close, and such integers can be factored very quickly using classical difference of squares algorithms. Apart from such aberrations in estimation, we know $0 < \delta \leq \frac 12$, and the size of $\delta$ directly suggests that the smallest prime dividing $N$ is of approximate order $N^\delta$. Then Theorem \ref{Thm: Primes in intervals} gives a nontrival reduction in the required search space for identifying $p$ by trial division.

We now consider the question of whether suitable estimations $a_{k,\ell}(N)$ might produce fruitful factorization algorithms in terms of computational complexity. To this end,
pick a positive real number $\varepsilon > 0$, and suppose by way of hypothesis we have an algorithm running in $O\lp \log^M(N) \rp$ operations which produces a value $a_{k,\ell}(N)$ such that, if $x_N, y_N$ are the minimal integers such that
\begin{align*}
    \dfrac{3}{4 x_N^k} < \dfrac{c_{k,\ell}(N)}{N^{k+\ell}} \ \ \ \ \ \text{and} \ \ \ \ \ \dfrac{3}{4y_N^k} < \dfrac{a_{k,\ell}(N)}{N^{k+\ell}},
\end{align*}
respectively, that $(1-\varepsilon) y_N < x_N < (1+\varepsilon) y_N$ for $N \gg_\varepsilon 0$. Then by Theorem \ref{Thm: Quick Factorization Range}, noting that $\omega(N) \leq 1/\delta$ for $N \gg 0$ because $p \asymp N^\delta$, we know that $N$ has its minimal prime divisors in the range
\begin{align*}
    \left[ (1-\varepsilon) y_N, \lp \dfrac{8}{3\delta} \rp^{1/k} (1+\varepsilon) y_N \right].
\end{align*}
By the prime number theorem, this interval contains
\begin{align*}
    \dfrac{(8/3\delta)^{1/k} (1+\varepsilon) y_N}{\log\lp (8/3\delta)^{1/k} (1+\varepsilon) y_N \rp} - \dfrac{(1 - \varepsilon) y_N}{\log\lp (1 - \varepsilon) y_N \rp} \sim \dfrac{y_N}{\log(y_N)} \lp\dfrac{8}{3\delta}\rp^{1/k}
\end{align*}
primes, for $N \gg_\varepsilon 0$ asymptotically. Now, by construction, we know by \eqref{Eqn: Wide interval eqn} that
\begin{align*}
    \dfrac{3}{4p^k} \leq \dfrac{3}{4 x_N^k} < \dfrac{c_{k,\ell}(N)}{N^{k+\ell}} < \dfrac{2/\delta}{p^k}.
\end{align*}
Thus, we must have $x_N^{-k} \asymp \delta^{-1} p^{-k}$, thus $x_N \sim y_N \asymp \delta^{-1/k} p \asymp \delta^{-1/k} N^\delta$. Therefore, $N$ must have its minimal prime divisor in the aforementioned interval, which contains 
\begin{align*}
    O\lp \dfrac{\delta^{-1/k} N^\delta}{\log(\delta^{-1/k} N^\delta)} \lp \dfrac{8}{3\delta} \rp^{1/k} \rp
\end{align*}
primes. Using a segmented Sieve of Eratosthenes, it will take $O\lp \delta^{-1/k} N^\delta \log\log\lp \delta^{-1/k} N^\delta \rp \rp$ operations to locate said primes. Therefore, when it is known that the minimal prime $p$ dividing $N$ has $p \asymp N^\delta$, we can determine this prime in
\begin{align*}
    O\lp \delta^{-1/k} N^\delta \log\log\lp \delta^{-1/k} N^\delta \rp \rp
\end{align*}
arithmetic operations. If a suitable approximation $a_{k,\ell}(N)$ to $c_{k,\ell}(N)$ can be computed sufficiently quickly, then the aforementioned procedure describes an algorithm for taking an integer $N$, about whose smallest prime factor we know nothing {\it a priori}, calculating quickly a value $\delta$ such that $N$ can be factorized deterministically in $O\lp N^\delta \log\log(N) \rp$ arithmetic steps, which requires $O\lp N^\delta \log^2(N) \rp$ bit operations\footnote{More pedantic improved bounds are of course possible using more iterated logarithms; such technicalities are not here our focus since this algorithm rests on the unknown question of approximating $c_{k,\ell}(N)$ quickly.}. Of particular note is that the best known deterministic factorization algorithms have runtime $O\lp N^{1/5 + o(1)} \rp$ \cite{Harvey,HarveyHitter}, and therefore if the coefficient $c_{k,\ell}(N)$ can be either calculated or suitably approximated, e.g., in $O\lp N^{1/5 - \varepsilon} \rp$ arithmetic operations, then the remaining discussion of this section provides a deterministic factorization algorithm  with worst-case runtime equal to that of \cite{HarveyHitter} but which improves their runtime whenever the minimal prime dividing $N$ satisfies $p = O\lp N^{1/5 - \varepsilon} \rp$. Such an algorithm would follow the following procedure:
\begin{enumerate}
    \item Calculate the value $$\delta := \dfrac{1}{k \log(N)} \log\lp \dfrac{3 N^{k+\ell}}{4 a_{k,\ell}(N)} \rp.$$
    \item If $\delta \geq \frac 15$, use the algorithm of \cite{HarveyHitter}.
    \item If $\delta < \frac 15$, find $p$ in the interval guaranteed by Theorem \ref{Thm: Quick Factorization Range} using a segmented Sieve of Eratosthenes and trial division.
\end{enumerate}
In light of this new discovery, we pose the following open question.

\begin{question}
    Is there a polynomial-time algorithm to compute an approximate value for $c_{k,\ell}(N)$ which is sufficiently accurate as $N \to \infty$ to implement the algorithm described above efficiently?
\end{question}

It is also interesting to question where the smallest prime $p$ dividing $N$ is likely to reside in the interval guaranteed by Theorem \ref{Thm: Quick Factorization Range}. Any predictable behavior in this realm, which is quite plausibly exhibited in the kinds of semiprimes used in the RSA algorithm, would then either yield an improved deterministic algorithm for integer factorization (again dependent on rapid asymptotic estimation of $c_{k,\ell}(N)$) or even improved probabilistic algorithms. If such an interval of confidence were sufficiently narrow, it may be the case that polynomial-time deterministic factorization is equivalent to polynomial-time deterministic approximation to $c_{k,\ell}(N)$. Whether or not an approximation algorithm to these coefficients exists which even improves upon currently known results, or provides a faster probabilistic factorization algorithm that improves on known methods, would be an interesting question for future work.

\section{Factorization Example} \label{Sec: Appendix B}

Consider $n=53,229,037$. The classical rough estimate for the smallest prime factor $p$ of $n$ is $2\leq p\leq \sqrt{n} \approx 7295.8.$ To make use of Theorem \ref{Thm: Quick Factorization Range}, we first check primes up to $\ln(n) \approx 17.79$ and verify that no such prime divides $n$. We then compute\footnote{For the sake of the example, we assume we know how to compute $c_{k,\ell}(n).$}, for example,
\begin{equation*}
    c_{3,7}(n) = 967357118954379126066831129989639443688588759415375453525761356800 \approx 9.67 \times 10^{65}.
\end{equation*}

\noindent By Theorem \ref{Thm: Quick Factorization Range}, the lower bound for the smallest prime dividing $n$ is $\lp 4c_{3,7}(n)/n^{10}\rp^{-\frac{1}{3}} \approx 5211.79$, and the upper bound is $\lp\lp 4c_{3,7}(n)/n^{10}\rp^{-\frac{1}{3}}+1\rp\lp3\omega(n)/8\rp^{\frac{1}{3}}\approx 9107.58$, assuming $n$ has 2 distinct prime factors, i.e., $5212 \leq p \leq 9107.$

The upper bound does not improve on the square root bound, but we can sharpen the bounds by considering larger values of $k$ and $\ell$. Consider $k=21$ and $\ell=23$ where again $n=53,229,037$. In this case, $c_{21,23}(n) \approx 3.696\times 10^{259}.$ Now, we deduce $6635\leq p \leq 7185$, which is sharper both in its upper and lower estimates and improves on the $\sqrt{n}$ bound of $7295.$ There are $42$ primes in this range and it can be checked quickly that $p=6737,$ i.e., $53,229,037=6737 \times 7901.$ Higher values of $k$ give increased opportunity to improve the search range at the cost of requiring estimates for very large values of $c_{k,\ell}(n)$. In such practical terms, as mentioned in Appendix A, it is natural to inquire whether estimates like $c_{21,23}(n) \approx 3.696\times 10^{259}$ can be derived rapidly by enumerating a relatively small number of partitions and using MacMahon-type formulas for estimation.

\section*{AI and computational resource disclosure}

The main results, their formulation, and the underlying ideas of proof are the authors' own. The second author used Gemini 3.6 Flash throughout as a research assistant, including as a tool for verifying certain numerical computations and as a dialogue partner to stress test arguments. All statements and proofs were verified by the authors, who take sole responsibility for them.  


\end{document}